\documentclass[article]{amsart}
\usepackage{amsthm,amsfonts}
\usepackage{amssymb,graphicx}
\usepackage[all]{xy}

\newtheorem{theorem}{Theorem}[section]
\newtheorem{lemma}[theorem]{Lemma}

\newtheorem{proposition}[theorem]{Proposition}
\newtheorem{definition}[theorem]{Definition}
\newtheorem{remark}[theorem]{Remark}
\newtheorem{conjecture}[theorem]{Conjecture}

\newcommand{\C}{\mathbb C}
\newcommand{\N}{\mathcal N}

\newcommand{\R}{\mathbb R}

\newcommand{\0}{{\bf 0}}
\newcommand{\x}{{\bf x}}

\newcommand{\bv}{{\bf v}}

\usepackage{graphicx}
\graphicspath{%
    {converted_graphics/}
    {/}
}
\begin{document}

\title[Bilipschitz geometry of real surface singularities  ]{ Bilipschitz geometry of real surface singularities whose tangent cone is a plane}

\author{Donal O'Shea}

\author{Leslie Wilson}




\address{Donal O'Shea (doshea@ncf.edu): New College of Florida, 5800 Bayshore Road, FL 34243, USA}

\address{Leslie Wilson (les@math.hawaii.edu): Mathematics Department, University of Hawai`i at Manoa, 2565 McCarthy Mall, Honolulu, HI 96822}



\begin{abstract}
Tangent cones are preserved under ambient bilipschitz equivalence, but the behavior of the Nash cone is more delicate.  This paper explores the behavior of the Nash cone and of exceptional rays under ambient bilipschitz equivalence for real surfaces in $R^3$ with isolated singularity and whose tangent cone is a plane.  
\end{abstract}

\maketitle

\section{Introduction}

Over the last two decades, bilipschitz geometry has proved to be particularly well adapted to the exploration of algebraic and semialgebraic singularities.  It gives rise to equivalence relations that 
distinguish some geometrically interesting features of singularities that analytic and topological equivalence obscure. 

Three distinct bilipschitz categories are of particular interest: inner, outer and ambient  bilipschitz equivalence of two subspaces of Euclidean space.  The first two arise from homeomorphisms between the subspaces that preserve the intrinsic distance function of the subspaces and the Euclidean distance function restricted to the subspaces, and the third from a bilipschitz homeomorphism between the ambient Euclidean spaces that preserves the two subspaces. We review the definitions in more detail in Section 3. The inner case is relatively simple; see \cite{B}. The outer case is much more difficult.  The difference between the outer and ambient cases is explored in \cite{BG1}.

In \cite{BFGG}, the bilipschitz contact class of a real-valued function germ $f$ on $\R^2,0$ is studied by dividing the domain into finitely many simple pieces with the property that the order of the restriction of $f$ to the germs of arcs at $0$ on each piece is affinely related to the ``width" of the set of arcs for which $f$ has that order.  The set of germs of arcs at a point in a space was studied by Valette \cite{V}. It often is useful to study ``zones" which are ``connected" subsets of this ``Valette link" that need not be the set of all arcs in a subset of our domain. These methods are used to give a classification of outer bilipschitz equivalence of surfaces using tools developed by Birbrair, Gabrielov and others (see \cite{BFGG},   \cite{BG2}, \cite{BG3}, \cite{BG4}, \cite{GS}). In \S 5, we reprise briefly the techniques we need for our results.  

In \cite{S}, Sampaio shows that two semialgebraic sets that are outer bilipschitz homeomorphic have outer bilipschitz homeomorphic Zariski tangent cones.  
The Nash cone of a semialgebraic set captures much more of the finer geometric structure of a semialgebraic set in the neighborhood of a singularity than the Zariski tangent cone, and it is natural to ask whether the analogue of Sampaio's result holds for the Nash cone of a semialgebraic set.  In \cite{OW1}, we showed that this was not the case, but were unable to characterize the behavior of the Nash cone in some natural cases, such as when the Zariski tangent cone is a plane. If $X$ is a semialgebraic surface germ at $0$ in $\R^3 \equiv \R^2 \times \R$ which is the graph of a continuous function germ on $\R^2$ at $0$ and whose tangent cone at $0$ is $\R^2$, we conjectured (Conjecture 4.3 of \cite{OW1}) that $X$ is ambient bilipschitz homeomorphic to $\R^2$ iff the inner and outer metrics on $X$ are equivalent (that is, if an only if $X$ is Lipschitz normally embedded). In this paper, we use versions of the tools mentioned in the previous paragraph to prove our main theorem, Theorem 6.6, which establishes our conjecture under a mild additional hypothesis on $X$.  In Section 2 we go into more detail on tangent and Nash cones, and our results from \cite{OW}.   In Section 6 we prove the main theorem modulo a result on closedness of certain zones.  The arguments through this point are very elementary.  In section 8, we require a deep result, the preparation theorem for subanalytic sets, to prove the closedness result.

Although it seems plausible that our main theorem might be derived from the outer lipschitz classification mentioned above, we have not seen how to do it.  While our argument uses some very similar tools, it is rather different.  Instead of zones and order relationships for lipschitz functions on $X$ (which will often be distance functions between different parts of $X$), we use zones and order relations defined from certain slopes of $X$ over arcs in $\R^2$. In the cases of most interest these slopes may approach or even equal infinity along a whole arc, so they are not lipschitz functions and, in fact, are not continuous at $0$.   These slopes determine the limit of the tangent planes to $X$ as we approach $0$ on the arc, and thereby determine the Nash cone. 

In Section 7, we use our techniques to analyze three examples.  While direct, the computations are subtle, and we think that results will be useful for readers. We were certainly surprised at the rich structure that even simple equations could exhibit.  

Our main results all deal with local behavior near 0. They are often stated in the semialgebraic category, but should be valid for the subanalytic case.

\section{Zariski and Nash cones}

Let $V\subset\R^3$ be a semialgebraic surface containing the origin ${\bf 0}\subset\R^3$. Two natural semialgebraic sets, the (Zariski) tangent cone and the Nash cone, reflect the local geometry of $V$ at $\0$.  The tangent cone, denoted $C \equiv CV \equiv C^+(V, {\bf 0})$, is the set of tangent vectors: that is, $\bv \in C$ if and only it there exist $\x_n \in V - \{\0\}, \x_n \rightarrow \0$ and a sequence of positive real numbers $t_n > 0$ such that $t_n\x_n \rightarrow \bv$.  The Nash cone,  denoted $\N\equiv \N V \equiv \N(V, \0)$, is the set of 2-planes $T$ with the property that there exists a sequence $\{ \x_n\}$ of smooth points of $V$ (that is, points at which $V$ is locally a 2-dimensional $C^1$ manifold) converging to $\0$ such that $T$ is the limit of tangent spaces to $V$ at the points $\x_n$. 

It follows from results of Whitney \cite{W} that $\N(CV, \0) \subset \N (V, \0)$. However, this containment may be (and frequently is) proper, reflecting the fact that the Nash cone captures more of the local geometry of $V$ than the Zariski tangent cone.  In \cite{OW}, we establish a structure theorem for the Nash cone analogous to one proved by L{\^ e} and Teissier for complex analytic surfaces \cite{LT}.

To describe it, note that if $\{ \x_n \in V\}$ is a sequence of points on $V$ approaching the origin, we may, by passing to a subsequence assume that the sequence $\{\x_n\}$ approaches the origin tangent to some ray $\ell$.  Necessarily, $\ell \subset C$.  Let $\N_\ell(V, \0) \subset \N(V,\0)$ denote the space of limits of tangent spaces that can be obtained as limits of tangent spaces along sequences tending to the origin tangent to $\ell$.  (Whitney's arguments in \cite{W}  show that if $V$ is semialgebraic, then $T\in \N_\ell$ implies $\ell \subset T$.) The arguments of \cite{OW} show that if $V \subset \R^3$ is a reduced, semialgebraic surface with $\0 \in V$, then there exist finitely many rays 
$\ell_1, \ldots, \ell_r$ in $C$, called {\it exceptional rays}, with $\N_{\ell_i}$ connected, closed and one-dimensional.  For any other ray $\ell \in C - \{ \ell_1, \ldots, \ell_r\}$, $\N_\ell(V, \0)$ is a single point (that is a single plane), and   $\N_\ell(V, \0) = \N_\ell(C, \0)$.  

An exceptional ray $\ell$ is said to be {\it full} if  $\N_\ell$ consists of the full pencil of planes in $\R^3$ containing $\ell$.  In the case of complex analytic surfaces, all exceptional lines are full, so that knowledge of the tangent cone and exceptional rays completely characterizes the Nash cone. For real surfaces, the exceptional lines need not be full (for examples, see \cite{OW} and \cite{OW1}), so for each exceptional ray, one needs an additional parameter encoding the closed, connected, one-dimensional set $\N_\ell (V, \0 )$ in order to recover $\N$.  The structure theorem in \cite{OW} is stated for real algebraic surfaces in $\R^3$; that is for surfaces given implicitly by an equation 
$\{ f= 0\}$ where $f\in \R[x,y,z]$ is a polynomial vanishing at $\0$.  However, the techniques and results apply without change to semialgebraic surfaces; indeed, locally near a point (say $0$) they also apply to subanalytic surfaces.

\section{Bilipschitz geometry}

A map $h: V \rightarrow W$ between two metric spaces $(V, d_V)$ and $(W, d_W)$ is said to be {\it lipschitz} if $$d_W(h(x), h(y)) \le Kd_V(x, y)$$ for all $x,y \in V$ and some constant $K>0$, and {\it bilipschitz} if $h^{-1}$ exists and is lipschitz.  Equivalently, $h: V\rightarrow W$ is bilipschitz if and only if $h$ is surjective and there exists $K> 0$ such that 
$$\frac{1}{K} d_V(x, y) \le d_W(h(x), h(y)) \le Kd_V(x, y).$$
A semialgebraic set $V$, real or complex, embedded in $\R^n$ or $\C^n$  has two natural metrics. One, the {\it intrinsic} or {\it inner} metric is the metric on $V$ induced by defining the distance 
$d_i(x,y)$ between two points $x, y \in V$ to be the 
infimum of the lengths of piecewise analytic arcs on $V$ joining $x$ and $y$.  The {\it outer} metric on $V$ defines the distance between any two points $x$ and $y$ to be their Euclidean distance $d_o(x,y)=|x-y|$ in the ambient space in which $V$ is embedded..  
Two such sets $V,W$ will be said to be inner (resp.\ outer) bilipschitz homeomorphic if they 
are bilipschitz homeomorphic with respect to the inner (resp.\ outer) metrics.  
If we don't say otherwise, we will mean outer.  In addition, if the outer bilipschitz 
homeomorphism is the restriction of a bilipschitz homeomorphism
 of the ambient Euclidean spaces 
 transforming $V$ into $W$, we will say they are ambient 
bilipschitz homeomorphic. 

A surface $V$  is said to be {\it normally embedded} (see \cite{BM}) if its outer and inner metrics are equivalent (that is, if there is a constant $K>0$ such that $d_i(x, y) \le Kd_o (x,y)$ for all $x,y\in V$.  In such a case, we say the $V$ is {\it length regular} or $\ell$-{\it regular}.  (In \cite{FW}, $\ell$-regularity is referred to as 1-regularity following \cite{T} (p. 79) where a hierarchy of regularity is defined.). Since outer bilipschitz equivalence implies inner bilipschitz equivalence, it is not hard to see that $\ell$-regularity is invariant under outer bilipschitz equivalence.

A homeomorphism is semialgebraic (subanalytic) if its graph is semialgebraic (subanalytic).
All our bilipschitz homeomorphisms are assumed to preserve the origin.

\section{The case when the tangent cone is a plane}

Let $\0 \in U \subset \R^2$ be semialgebraic (subanalytic) and open and $f:U \rightarrow \R$  be continuous and semialgebraic (subanalytic).  Suppose that the graph $X = \Gamma f$ is $C^2$-nonsingular except at $\0$ and that 
the tangent cone $C_\0 (X) = \R^2$.  The natural projection $\pi : X \rightarrow \R^2$ is singular for $x \in \Sigma (\pi)$ when $T_xX$ is vertical (that is, parallel to the $z$-axis) for $x\in \Sigma(\pi)$.

There are at most finitely many exceptional rays in $C_\0(X)$.  If none are full exceptional we know that
the germ at $X, \0$ is semialgebraic (subanalytic) bilipschitz to $\R^2, \0$. So, without loss of generality, we may assume that the positive $y$-axis is a full exceptional ray and we restrict $f$ to a wedge $\{ |x|\le ky, y\ge 0 \} $ intersected with $U$ that contains no other exceptional ray.  Call this $W$ and let
$X= \Gamma f|_W$.  

\section{Connected families of analytic arcs in $\R^2$}

We investigate $X$ above by studying the behavior of $f$ along analytic arcs (an analytic arc is a half-branch of an analytic curve $f(x,y)=0$) in the upper half-plane (usually restricted to $W$), parametrized by $ y\ge 0$.
Let $\mathcal{R}$ be the set of all such analytic arcs $\{x=\alpha(y), y\ge 0\}$; this is the Valette link of $0$ in the upper half-plane (or in $W$). Every element of $\mathcal{R}$ has a Puiseux expansion
$$\alpha(y) = c_1y^{p_1} + h.o.t.$$
where $p_1$ is a positive rational 
number, $c_1>0$, and $h.o.t$ denotes non-zero higher order rational 
terms.  
Define the {\it order} 
$\mathcal{O}(\alpha)$ of $\alpha$ to be the smallest nonzero power $p_1$ in its Puiseux expansion.

If $\alpha, \beta \in \mathcal{R}$, we write $\alpha < \beta$ if $\alpha(y) < \beta(y)$ for all sufficiently small nonzero $y$.  

A subset $A\subset \mathcal{R}$ is called {\it connected} if for all pairs of arcs $\alpha, \beta \in A$ with $\alpha < \beta$, and any 
$\gamma$ such that $\alpha < \gamma <\beta$, then $\gamma \in A$.  A connected subset of $\mathcal{R}$ is often called a {\it zone} in the literature.  

The {\it distance}, denoted $d(\alpha, \beta)$ between two arcs $\alpha(y), \beta(y) \in \mathcal{R}$ is defined to be the reciprocal of the order of their difference:
$$
 d(\alpha, \beta) = \frac{1}{\mathcal{O}(\alpha(y) - \beta(y))}.
 $$

The {\it width} $w(A)$ of a zone $A \subset \mathcal{R}$ is defined to be 
$$
w(A) = \sup\big\{d(\alpha, \beta) \ \ {\rm where}\ \  \alpha, \beta \in A \big\} .
$$  

A zone $A$ is called {\it closed} if there exist $\alpha, \beta \in A$ such that 
$$ 
 w(A) =  d(\alpha, \beta).
 $$
 Otherwise, $A$ is called {\it open}.  

 \vskip .25in
 
\noindent{\bf Example.}  Let $$A = \{ x = y + cy^2 + {\rm all\  }h.o.t., \ \ 0 \le c < 1\}$$ and 
 $$B = \{x = y + y^2 + {\rm all\  }h.o.t. \}.$$
 Then $w(A)= w(B) = \frac{1}{2}$.  Note that $A$ is closed (let $\alpha(y) = y + \frac{1}{2}y^2 + h.o.t.$ and $\beta(y) = y + \frac{1}{3}y^2 + h.o.t.$), while $B$ is not.  To see the latter,
 note that if
 $ \alpha, \beta \in B$, then $\mathcal{O}(\alpha(y) - \beta(y)) > 2,$
 and we can choose $\alpha, \beta$ making $\mathcal{O}(\alpha(y) - \beta(y))$ arbitrarily close to 2, so that $w(B) = \frac{1}{2}$.
 Let 
 $$C = \{x = y + y^2 + cy^r +{\rm all\  }h.o.t. , \ \ {\rm for\  all\  } r>2, c>0 \}, $$
 which consists of all members of $B$ lying strictly below (as germs) the arc $x = y + y^2 $.
 Then $w(C)=2$ is $C$ is open by the same argument as for $B$.
 \vskip .25in
 
 \begin{lemma}  Suppose that $A$ is a zone (hence, connected) with $w(A) = w$.  Then there is a unique finite Puiseux series 
 $$x = p(y) = a_1y^{r_1} + \ldots + a_ry^{r_n}$$ with $r_i$ positive rational numbers $r_1 < \ldots < r_n < \frac{1}{w}$ such that
 each $\alpha \in A$ has the form
 $$\alpha = p(y) + h.o.t._{\ge \frac{1}{w}} $$
 ($h.o.t._{\ge \frac{1}{w}}$ denotes a sum of non-zero terms of order greater than or equal to $\frac{1}{w}$). 
\end{lemma}
\begin{proof}  Suppose $\alpha, \beta \in A$ with $\alpha = p(y) + h.o.t_{\ge \frac{1}{w} }$ and
$\beta= q(y) + h.o.t_{\ge \frac{1}{w} }$ with $p(y) \ne q(y)$.  Then $p(y) - q(y)$ has a first term with nonzero coefficient and order less than $w$ which implies $w(A) > w$, a contradiction. 

\end{proof}

\begin{lemma}  If $A$ is a closed zone with $w(A) = w$, there exists a finite Puiseux series
$$x = p(y) = a_1y^{r_1} + \ldots + a_ry^{r_n}$$
 with $r_i$ positive rational numbers $r_1 < \ldots < r_n < \frac{1}{w}$,  real 
 numbers $c_1 < c_2$, and two Puiseux series
$$\alpha_1 = p(y) + c_1y^{\frac{1}{w}} + h.o.t._{> \frac{1}{w}} \in A$$
$$\alpha_2 = p(y) + c_2y^{\frac{1}{w}} + h.o.t._{> \frac{1}{w}}  \in A.$$
Moreover, for all $c$ such that $c_1 < c < c_2$, all Puiseux series
$$\alpha = p(y) + cy^{\frac{1}{w}} + h.o.t._{> \frac{1}{w}} $$
lie in $A$.
\end{lemma}

\begin{proof}Since $c_1 < c < c_2$ and all other terms of order less or equal $1/w$ with non-zero coefficients are identical
in $\alpha_1, \alpha, \alpha_2$, we have $\alpha_1 < \alpha <\alpha_2$.  Since $A$ is connected,
$\alpha \in A$.
\end{proof}

\begin{lemma}  If $A$ is an open zone with $w(A) = w$, there exists a finite Puiseux series $p(y)$ with order less than $\frac{1}{w}$ and a 
constant $c \in \R$ such that every $\alpha \in A$ has the form
$$
\alpha = p(y) + cy^{\frac{1}{w}} + h.o.t.
$$
and for every rational  number $r > \frac{1}{w}$ there is an interval $I$ such that every $$\alpha = p(y) + cy^{\frac{1}{w}} + dy^r +h.o.t.$$
with $d\in I$ is in $A$.  
\end{lemma}

\smallskip

\begin{definition}  Two zones $A, B$ are said to be {\it adjacent} if $A\cup B$ is connected and $A\cap B = \emptyset $.
\end{definition}

Note that the zones $A$ and $B$ of width $\frac{1}{2}$ in the example above are adjacent.  

\begin{lemma} \label{openclosed} Suppose that $A, B$ are adjacent zones. 
\begin{enumerate}
\item If $w = w(A)\ge w(B)$, then $w(A \cup B) = w$.
\item If $A$ is closed and $w = w(A)\ge w(B)$, then $A \cup B$ is closed.
\item If $A$ is open and $w(A) > w(B)$, then $A \cup B$ is open.
\item If $A$ is closed and $B$ is open, then $w(A)> w(B)$.
\end{enumerate}
\end{lemma} 
\begin{proof}  Straightforward.
\end{proof}
 
 \vskip .25in
 \begin{lemma}
 Suppose $A, B, C$ are adjacent zones (in that order) and that $A, C$ are closed.  Then $A \cup B \cup C$ is closed.  
 \end{lemma}
 \begin{proof} If $B$ is closed, the conclusion follows from the preceding lemma.  In addition, if
 $w(B) \le \max(w(A), w(C))$ then the conclusion again follows from the preceding lemma.  So, suppose $B$ is open
 and $w=w(B) > \max(w(A), w(C))$; we will show this leads to a contradiction, thereby proving the Lemma.  There is a finite Puiseux series $p(y)$ and a number $b$ such that all elements of $B$ are
 of the form $p(y) + by^{\frac{1}{w}} +  h.o.t$, with $w= w(B)$.  By adjacency, all elements of $A$ and $C$ are of the form
 $$p(y) + by^{1/w} + ay^{r_A} + \ldots , \quad r_A > \frac{1}{w(B)}$$
 $$p(y) + by^{1/w} + cy^{r_C} + \ldots, \quad r_C > \frac{1}{w(B)}. $$
 Choose $r$ such that 
 $${1/w} < r < \min(r_A, r_C)$$ 
 and an interval $I$ so that 
 $$ p(y) + by^{1/w} + dy^r$$
 are in $B$ for all $d \in I$.  But these don't lie between $A$ and $C$, a contradiction.  
 \end{proof}

 \section{Increasing, decreasing, and flat zones and their heights}
 
We now turn to the behavior of $f$ on subsets of arcs in $\mathcal{R}$.  By the assumptions we made at the outset of section 4, we are examining the graph $X$ of a function $z= f(x, y),\  f(0,0) = 0$, with tangent plane $C(X, \0)$ the $xy$-plane, and the Nash cone 
$\N(X, \0)$ consisting of all planes containing the $y$-axis.  The limit of $f_y$ along an arc $\gamma(y) \in \mathcal{R}$  will determine the tangent plane in  $\N(X, \0)$ to which $f_y(\gamma(y))$ tends. 
In this section we will state all the results in the semialgebraic case, but they also hold in the subanalytic case.

We say that a subset of $\mathcal{R}$ is $FI$ ({\it fast increasing}), $FD$ ({\it fast decreasing}) or $FL$ ({\it flat}) according as:
\begin{eqnarray*}
FI &:& \{\gamma\ | \frac{\partial f}{\partial x} \circ\gamma > 0  
{\rm \ and}\ \rightarrow\ \infty\  {\rm as}\ y \rightarrow 0\}\\
FD &:& \{ \gamma\ | \frac{\partial f}{\partial x} \circ\gamma < 0 
{\rm \ and}\ \rightarrow\ -\infty\  {\rm as} \ y \rightarrow 0\}\\
FL &:& \{ \gamma\ | \frac{\partial f}{\partial x}\circ\gamma  \not\rightarrow \infty \ {\rm as}\ y \rightarrow 0\}
\end{eqnarray*}
We include in $FI$ those arcs $\gamma$ for which $\frac{\partial f}{\partial x} \circ\gamma = \infty$ and $\frac{\partial f}{\partial x}  > 0$  near $\gamma$ on both sides (in this case we say  $\frac{\partial f}{\partial x} \circ\gamma = +\infty$ ). Similarly we include in $FD$ those arcs $\gamma$ for which $\frac{\partial f}{\partial x} \circ\gamma = -\infty$.
\noindent By taking unions of connected subsets of arcs that are $FI, FD, FL$, we can (and will) assume $FI, FD, FL$ regions are ``maximally connected" zones.



In every $FI$ or $FD$ zone there is an arc
along which $f_{xx} =0$ or $f_x = \infty$; inside every $FL$ zone there is at
least one arc on which $f_{xx} =0$ (if the $FL$ zone connects two $FI$ or two $FD$
zones) or an arc with $f_x =0$ (if the $FL$ zone connects an $FI$ with an
$FD$).  The points where these derivatives hold form a semialgebraic set $S$,
and so the number of connected components  of $S-\{0\}$ is finite. Thus the numbers of maximally connected flat
zones and maximally connected FI and FD zones are finite.  The proposition below follows.

\begin{proposition}
$\mathcal{R}$ is the finite disjoint union of all the maximally connected $FI$, $FD$ and $FL$ regions.  Each $FI$ and $FD$ region is preceded and followed by an $FL$ region.
\end{proposition}

Note that for every arc $\alpha (y)$, we have $\frac{\partial f}{\partial x} (\alpha (y))\rightarrow \ell$ for some $\ell \in [-\infty,\infty ]$.  In particular, we allow $\frac{\partial f}{\partial x}(\alpha(y)) \rightarrow \infty$.

\begin{lemma} Given an arc $\alpha$ as above, there exists a semialgebraic  open set $U$ containing $\alpha$   such that if an arc $\beta(y)$ lies in $U$, then $\frac{\partial f}{\partial x}(\beta(y)) \rightarrow \ell$.  In particular, this holds for $\ell = \infty$.
\end{lemma}

This lemma could be avoided, but it is convenient.  For example, it follows that if $A$ is $FI$ (resp. $FD$, resp. $FL$) and $\alpha \in A$, then there exist $\beta_1, \beta_2 \in A$ with $\beta_1 < \alpha < \beta_2$.

\begin{proof} The partial derivative of a semialgebraic function is semialgebraic (this is  Exercise 2.10 in \cite{C}, with solutions to the exercise found online).  So $g(x,y) = f_x(x,y)$ is a semialgebraic function and we're given an arc $\alpha(y)$ with $g(\alpha(y))$ going to
\ $\ell$\ as $y$ goes to $0$.  Then there is a semialgebraic positive function $h(y)$ which goes to $0$ as $y\rightarrow 0$ and 
$\| g(\alpha(y))  - \ell \|< h(y)$.  So $U = \|g(x,y)- \ell \|<h(y)$ is a semialgebraic set containing $\alpha$, and any arc $\beta \in  U$ also satisfies $f_x(\beta(y)) \rightarrow \ell$, proving the lemma.  \end{proof}

\begin{definition} For any $A$ that is an $FI, FD$ or $FL$ region, define the 
height $h(A)$ to be
$$
h(A) = \sup \left\{ \frac{1}{O(f(\alpha(y))-f(\beta(y))) }, \alpha, \beta \in A\right\}.
$$
\end{definition}
\noindent Again, this term is the reciprocal of that used by Birbrair et al. (see \cite{BG1}, \cite{BG2}).

\begin{proposition}
a)  If $A$ is $FL$, then $h(A) \le w(A)$.\\
b) If $A$ is $FI$ or $FD$, then $h(A) \ge w(A)$.
\end{proposition}

\begin{proof}
Choose any two arcs $\alpha, \beta \in A, \alpha < \beta$.  Let
\begin{eqnarray*} 
\Delta x(y) &=& |\beta(y) - \alpha(y)| \ {\rm and}\\
\Delta z(y) &=& |f(\beta(y), y) - f(\alpha (y), y)|.
\end{eqnarray*}
By the mean value theorem, for each $y$, there exists an $x_y$ between $\alpha (y)$ and
$\beta (y)$ such that
$$\frac{\Delta z}{\Delta x} = |f_x(x_y, y)|.$$
If $A$ is flat, then $|f_x|$ is bounded on the region between $\alpha$ and $\beta$ and so
$\frac{\Delta z}{\Delta x}$  is bounded which means $\mathcal{O}(\Delta z(y)) \ge \mathcal{O}(\Delta x(y))$. Since this holds for all pairs 
$\alpha, \beta \in A$, we have $h(A) \le w(A)$.  If $A$ is $FI$ or $FD$, then 
$|f_x(x_y, y)| \rightarrow \infty,$ which means $\mathcal{O}(\frac{\Delta z}{\Delta x}) < 0$ 
and hence $\mathcal{O}(\Delta z) < \mathcal{O}(\Delta x)$.  Taking the $\sup$ over $A$, $h(A) \ge w(A)$.  (We don't have an example where it isn't strict inequality).
\end{proof}

Since we assume our $\Gamma f$ is assumed tangent to the plane at $0$, every $f(\gamma (t)) $ has order greater than 1; since our wedges $W$ are closed, we can conclude:

\begin{proposition}  Assuming $\Gamma f|W $ is tangent to $R^n$ at $0$, then $h(W)<1$.

\end{proposition}

\begin{conjecture} Maximal $FL$ zones are closed.  
\end{conjecture}

In \S 8,  we prove this conjecture under the proviso that there is   an analytic function $F(x,y,z)$ on a neighborhood of $0$, $X$ is subanalytic, $C^1$ and $X\subseteq F^{-1}(0)$,  $\dim{(F^{-1}(0) - X)}\le 1$, and $\pi: X\rightarrow \R^2$ is a homeomorphism near 0. 


\begin{theorem}
If the flat zones of $X=\Gamma f|_W$ are closed (so if the hypotheses in the paragraph above attain), then $X$ is ambient semialgebraic bilipschitz to $W$ if and only if $X$ is normally embedded (that is,
1-regular).  In particular this holds for $W=\R^2$. 
\end{theorem}
\begin{remark} This is conjecture 4.3 of \cite{OW1}, except we don't assume flat regions are closed there. \end{remark}
\begin{proof}
If $X$ is outer semialgebraic bilipschitz to $W$, then it is immediate that $X$ is 1-regular since $W$ is.

For the converse, assume $X$ is 1-regular. Then so is its restriction $X =\Gamma f$ over $W=\{ |x| \le ky, y \ge 0\}$ with the positive $y$-axis the only
exceptional ray.  The domain is partitioned into maximal connected $FI, FD$ and $FL$ regions which we list from left to right starting with $F_{\rm first}$ flat with $w(F_{\rm first}) = 1$ and ending with flat $F_{\rm last}$ with $w(F_{\rm last}) = 1$.  Next after $F_{\rm first}$ is a region $I_1$ (say, without loss of generality, $FI$), and then alternating $FL$ and $FI$:
$$
F_{\rm first}, I_1, F_1, I_2, F_2, \ldots, I_r, F_r\quad r\ge 1
$$
and either $r = {\rm last}$ or $F_r$ is followed by 
$$
D_1, F_{r+1}, D_2, \ldots, D_s,  F_{r + s}, {\rm etc.}  
$$
Here the $I$'s are $FI$ and $D$'s are $FD$.
 \newline
 
 By a modification of Theorem 4.1 of \cite{OW1}, we have:
\begin{proposition}\label{notregular}  Assume three successive regions $A,B,C$ are such that  $A$ is FI (resp. FD), $C$ is FD (resp. FI) and $B$ is either FI or FD or a succession of alternating flat and FI or FD regions. Suppose $h(A)> w(B)$ and $h(C)> w(B)$.  Then $\Gamma f|_{(A\cup B\cup C)}$ is not 1-regular.
\end{proposition}

Here we use Lemma \ref{openclosed} to conclude that $h(A ) >w(A\cup B\cup C)$ and   $h( C ) >w(A\cup B\cup C)$, and then apply the proof of Theorem 4.1 of \cite{OW1}.
\newline

By a modification of Theorem 4.2 of \cite{OW1}, we have:

\begin{proposition} \label{wellseparated} Assume three successive regions $A,B,C$ are such that $A$ and $C$ are flat and $B$ is either FI or FD or a succession of alternating flat and FI or flat and FD regions. Suppose $w(A)\ge h(B)$ and $w(C)\ge h(B)$.  Such an $A,B,C$ is called {\rm well-separated}.
Then there is a semialgebraic bilipschitz transformation of $\Gamma f|_{(A\cup B\cup C)}$ onto a flat $D$ (the linearization of $\Gamma f|_{(A\cup B\cup C)}$) which agree along their boundary, and which is such that 
the support of the transformation lies in a compact region $K$ containing both of these sets, and $\pi K=A\cup B\cup C$, where $\pi$ is orthogonal projection to the $xy$-plane.
\end{proposition}

Since $A$ is flat, hence closed, it can be split into two consecutive flat regions $A=A_1\cup A_2$ with $w(A_1)=w(A_2)=W(A)$, and similarly for $C$. The argument from Theorem 4.2 of \cite{OW1} can be applied to $A_2,B,C_1$, 
and the support of the semialgebraic bilipschitz transformation is compact and lies over $A_2,B,C_1$.  So $\Gamma f$ is unchanged over $A_1$ and $A_3$.  This will allow us to piece together the transformations created for each successive region.

By this proposition any well-separated region can be replaced by a single flat region.  This process can be continued until  there are no well-separated regions (NWSR).  Assume that is the case.

If $X$ is  not 1-regular, then it is not bilipschitz to $W$.  So assume $X$ is 1-regular.  
Setting $A_1 = \{ I_1, F_1, I_2, F_2, \ldots, I_r\},$ note that 
\begin{eqnarray*}
h(A_1) &=& \max\{h(I_1), \ldots, h(I_r), h(F_1), \ldots, h(F_r)\}\\
w(A_1) &=& \max\{w(I_1), \ldots w(I_r), h(F_1), \ldots, h(F_r)\},
\end{eqnarray*} 
and, so on, mutatis mutandis, for $h(A_i)$ and $w(A_i)$.  

By our NWSR assumption there must be an  $I_{r_0}$, $1\le {r_0}\le r$, 
with $h(I_{r_0})>w(F_r)$, and we assume this is the last such $I_{r_0}$ 
(i.e. $h(I_j)\le w(F_r)$ for ${r_0}<j\le r$ ).   Then $h(I_{r_0})> w(F_{r_0}\cup I_{{r_0}+1}\cup \dots \cup F_r)=w(F_r)$. 
Suppose there is an $s_0$, $1\le {s}$, such that   
 $h(D_{s_0})>w(F_r)$, and we assume this is the first such $I_{s_0}$ 
 (i.e. $h(D_j)\le w(F_r)$ for $1\le j <s_0$ ).  Arguing as above, 
 $h(D_{s_0})> w(F_r)=$ the width of the union of zones between 
 $I_{r_0}$ and $D_{s_0}$.  By  Proposition \ref{notregular} 
 this violates the assumption that $X$ is 1-regular.  
 Thus  $h(D_j)\le w(F_r)$ for $j=1, \ldots, s$. 
 By NWSR $w(F_{r+s})< w(F_s)$.   
Continuing this process gives that $w(F_{last})<1)$.  
This contradicts $w(F_{last})=1$. So $\Gamma f|W$ must be 1-regular.
\end{proof} 

\begin{remark} In the theorem, the bilipschitz transformation has compact support. \end{remark}

\section{Examples}

Given an arc $\alpha(y)$, recall that $\mathcal{O}(f_x|_\alpha)$ is the order of $f_x|_\alpha$.
We will sometimes distinguish between two types of $FL$  zones:
\begin{enumerate}
	
	\item  $FL_1 = \{\alpha: f_x|_\alpha \rightarrow 0\ \  {\rm as}\ \ y \rightarrow 0\}=\{\alpha:\mathcal{O}(f_x|_\alpha)>0\} $
	
	\item $FL_2 = \{\alpha: f_x|_\alpha \rightarrow {\rm nonzero \ constant}\ \  {\rm as}\ \ y \rightarrow 0\}=\{\alpha:\mathcal{O}(f_x|_\alpha)=0\} .$
	
\end{enumerate}\noindent

{\bf Example 1}   (3.1 of \cite{OW1})

$$
z = f(x, y) = \frac{y^{2a+b}}{x^2 + y^{2a}}, \quad b\ge 2, \  \ a > b \ {\rm positive\ integers}.
$$
The positive $y$-axis is a full exceptional line.  

\vskip .25in 

We have 
$$
f_x = \frac{\partial z}{\partial x} =  -\frac{y^{2a+b}\cdot 2x}{(x^2 + y^{2a})^2}\
$$
\noindent and consider the arcs 
$$x = \phi(y) = cy^s + {\rm h.o.t.} .$$
On these arcs
$$
f_x|_\phi = -y^{2a+b}(2cy^s + \ldots)/(c^2y^{2s} +\rm h.o.t. + y^{2a})^2.
$$


\noindent{\bf Case $0 < s < a$}: 

$$
f_x|_\phi = ky^{2a + b -3s} + \ldots \quad ({\rm where}\ k = - 2/c^3)
$$  
so
\begin{eqnarray*} 
f_x|_\phi &\rightarrow& 0 \ {\rm if}\   s < (2a + b)/3,\\ 
&\rightarrow& {\rm constant} \ne 0 \ {\rm if}\  s = (2a+b)/3, \\ 
&\rightarrow& \infty \ \ {\rm if}\   (2a+b)/3 < s < a.
\end{eqnarray*}

\noindent{\bf Case} $s = a$:

$$
f_x|_\phi = ky^{b - a < 0} + \ldots \rightarrow \infty \quad ({\rm where}\ k = - 2c/(c^2+1)^2).
$$  

\noindent{\bf Case} $a< s$:

$$f_x|_\phi = ky^{b - 2a + s} + \dots ({\rm where}\ k = - 2c)\ $$
so
\begin{eqnarray*}
f_x|_\phi &\rightarrow& \infty   \ {\rm if}\  a< s < 2a- b \\ 
&\rightarrow& {\rm constant} \ne 0 \ {\rm if}\  s = 2a- b \\ 
&\rightarrow& 0 \ {\rm if} \ 2a-b < s.
\end{eqnarray*}

So, in the first quadrant (the second quadrant is similar), the $FL_1$ zone consists of a zone of arcs 
$$\{\phi =  cy^s + {\rm h.o.t.}, {\rm for\ all}\ c>0\  {\rm and}\  0<s< \frac{2a + b}{3}\}$$   which is closed of width 1, and a zone  of arcs $$\{\phi =  cy^s + {\rm h.o.t.}, {\rm for\ all}\ c>0\  {\rm and}\ 2a -b < s \}$$  which has width $\frac{1}{2a-b}$ but which is open (note when $c<0$ the arc is in the second quadrant). 

The $FL_2$ zone consists of a zone of arcs of the form 
$$\{\phi =  cy^s + {\rm h.o.t.}, {\rm for\ all}\ c>0\  {\rm and}\ s = \frac{2a + b}{3}\}$$ 
and another of the form
$$\{\phi =  cy^s + {\rm h.o.t.}, {\rm for\ all}\ c>0\  {\rm and}\ s = 2a - b\}.$$
The former has width  $\frac{3}{2a + b}$ and the latter $\frac{1}{2a - b}$, and
both are closed.

\vskip .25in

The $FD$ zone for $f$ in the first quadrant is 
$$
x= cy^s + \ldots \quad {\rm with}\ c>0 \quad \frac{2a+b}{3} < s < 2a - b
$$ and has width $\frac{3}{2a + b}$.
Then
$$
z = y^{2a + b}/(c^2y^{2s} + \ldots + y^{2a})
$$
When $s = (2a+b)/3$
\begin{eqnarray*}
z &=& y^{2a+b}/(b^2y^{(4a + 2b)/3} + \ldots )\\
&=&\frac{1}{b^2}y^{(2a+b)/3} + \ldots,\\
{\rm so}\quad  \mathcal{O}(z) & = &  (2a+b)/3.
\end{eqnarray*}
When $s = 2a -b$
\begin{eqnarray*}
z &=& y^{2a+b}/(y^{2a} + \ldots)\\
&=& y^b + \ldots\\
{\rm so} \quad  \mathcal{O}(z) &=& b.
\end{eqnarray*}  

We have 
$$
\Delta z =y^b - \frac{1}{b^2}y^{(2a+b)/3} + \ldots \approx y^b + \ldots,
$$
so $h(FD) = \frac{1}{b}$.

\vskip .15in

Going to the right from the $y$-axis (so that $s$ decreases), the zones in the open first quadrant are 
$$
FL_1, FL_2, FD, FL_2, FL_1.
$$

The corresponding ranges for $s$ decrease as follows
$$
\{s>2a-b\}, \{s=2a-b\}, \{2a-b>s>\frac{2a+b}{3}\}, \{s = \frac{2a+b}{3} \}, \{s < \frac{2a+b}{3} \}
$$
and the corresponding widths of the zones are
$$
\frac{1}{2a-b}, \frac{1}{2a-b}, \frac{3}{2a+b}, \frac{3}{2a+b}, 1.
$$ 
Combining the first two zones gives $w(FL) =\frac{1}{2a-b}$ for the first flat zone, followed by a $FD$ zone with height
$h(FD) = \frac{1}{b}$ (and $w(FD) = \frac{1}{\frac{3}{2a+b}}$), which is followed by a third zone, which is $FL$ and which has width $w(FL) = 1$, obtained by amalgamating the last two zones.  So, going right from the $y$-axis, we have three zones
$$FL, FD, FL$$
with
$$w(FL) = \frac{1}{2a-b},\quad h(FD) = \frac{1}{b},\quad w(FL) = 1,$$
respectively.  

Finally, we combine with the second quadrant to get a zone of width $\frac{3}{2a+b}$ containing successive $FI$, $FL$ and $FD$ regions with height 
$\frac{1}{b}$ strictly greater than the width $\frac{3}{2a+b}$.  Thus the graph is not 1-regular and therefore not bilipschitz to $\R^2$.

 We can represent this example (in the first quadrant) with a graph and an order function. The graph consists of two edges, $(0,a)$ and $(a,\infty)$ with each point representing a power $s$, and with a single vertex at $s=a$. Each point $s$ represents the zone $\{\phi =  cy^s + {\rm h.o.t.}\}$, $c>0$.  The order function $\mathcal{O}(f_x|\phi)$ is independent of $c$ and the higher order terms.  We can write it as  $\mathcal{O}(s)= 2a+b-3s$ on  $(0,a)$ and $\mathcal{O}(s)= b-2a+s$ on $(a,\infty)$.  Note both these affine functions equal $\mathcal{O}(a)=b-a$, 
 so $\mathcal{O}(s)$ is continuous and piecewise affine.
 
 The only ``exceptional power'' is the vertex $s=a$ where the two affine functions meet.  The exceptional power $s=a$ has no ``exceptional coefficient" because the order function does not depend on $c$. 
 
 Note that $f_x(y)=0$ on the positive $y$-axis, which we can represent on the graph by $s=+\infty$, and in this case we consider $\mathcal{O}(f_x|\phi) =+\infty$.  We can also extend to the second quadrant in which the arcs have $c<0$, yielding the graph with 4 edges and the vertices $s= a (c>0)$, $s= a (c<0)$, and $s= +\infty$.  We can add vertices representing the  positive and negative $x$-axes,
 and extend around the circle using the above zones with $y<0$.  Observe in this case that the graph has eight edges and eight vertices and the order function is continuous and piecewise affine.  The FL zones represent the connected components of  $\mathcal{O}(s)\ge 0$, which are closed subsets of the graph, and hence are closed zones.
 


\vskip .25in

\noindent{\bf Example 2} (3.2 of \cite{OW1})

$$
z^3 = (x^2 + y^2)(x^2 - y^3) = x^4 + x^2( y^2 - y^3) - y^5.
$$
\smallskip
\noindent Set $f(x,y) = z^{1/3}$ and restrict to $\{ y > 0 \cup (0,0) \}$.

\noindent Let $x = cy^s + \ldots,\  s>0,$  so that
\begin{eqnarray*}
z^3 &=& c^4y^{4s} + c^2y^{2s}(y^2 - y^3) - y^5 + \ldots,\\
&=& c^4y^{4s} + c^2y^{2s+2} - y^5 + \ldots
\end{eqnarray*}


\noindent The lowest order term $(c\ne 0)$ is
\begin{eqnarray*}
{\rm I.} &&c^4y^{4s} \quad {\rm if}\quad s < 1\\ 
{\rm II.}&&c^2y^{2s+2}\quad {\rm if}\quad  1 \le s \le 3/2\\
{\rm III.}&& -y^5\quad {\rm if}\quad  3/2 < s\\
\end{eqnarray*}
\noindent which gives the following cases:
\begin{eqnarray*}
{\rm I.}&& z= c^{4/3}y^{4s/3} + \ldots  \quad {\rm if}\quad s < 1\\
{\rm II.}&&z= c^{4/3}y^{(2s+2)/3} + \ldots \quad {\rm if}\quad  1 \le s \le 3/2 \\
{\rm III.}&&z = -y^{5/3} +\ldots \quad {\rm if}\quad  3/2 < s \\
{\rm IV.}&&z = (c^4 + c^2)^{1/3}y^{4/3}+\ldots\quad{\rm if}\quad s =1\\
{\rm V.}&&z = (c^2 - 1)^{1/3}y^{5/3} +\ldots\quad{\rm if}\quad 3/2 = s.
\end{eqnarray*} 

\noindent Differentiating gives
\begin{eqnarray*}
3z^2f_x &=& 4x^3 +2x(y^2 - y^3)\\
f_x &=& (4x^3 +2x(y^2 - y^3))/3z^2
\end{eqnarray*}
\noindent and we obtain the following where $s$ ranges as above.
\vskip .25in

Restrict to the first quadrant for simplicity, the ``first level'' of the tree consists of the three edges $0<s\le 1$, $1\le s\le 3/2$ and $3/2 \le s<\infty$; the order function on these edges is respectively $s/3$, $(2-s)/3$ and $s-4/3$. Note these agree at the vertices $\mathcal{O}(1)=1/3$ and $\mathcal{O}(3/2)=1/6$. The exceptional $s=3/2$ has an associated exceptional coefficient, namely $c=1$ (see case V above). In that case we have to calculate the order function for the zone
\begin{eqnarray*}
x&=& \phi(y) = y^{3/2} + by^s + \ldots {\rm where}\ s > 3/2\\
\ \ {\rm then} &&\\ 
z^3 &=& ((y^{3/2} + by^s +\ldots)^2 + y^2)((y^{3/2} + by^s+\ldots)^2 - y^3)\\
&=& (y^2 + \ldots)(y^3 + 2by^{3/2 + s} + b^2y^{2s} + \ldots -y^3)\\
&=& 2by^{s + 7/2} + \ldots\\ 
\ \ {\rm and} &&\\
z&=& (2b)^{1/3}y^{s/3 +7/6}+ \ldots
\end{eqnarray*}

\begin{eqnarray*}
{\rm Thus}\ \ 3z^2\frac{\partial z}{\partial x} &=& 4x^3 +2x(y^2 - y^3)\quad {\rm and}\ \ x = y^{3/2} + by^s + \ldots\\
&=&4(y^{9/2}+3y^{6/2}by^s + 3y^{3/2}b^2y^{2S} + b^3y^{3s}) + 2(y^{7/2} + \ldots)\\
&=&2y^{7/2} + \ldots
\end{eqnarray*}
\begin{eqnarray*}
{\rm whence}\ \  f_x = \frac{\partial z}{\partial x} 
  &=& C(y^{7/2} + \ldots)/(y^{s/3 +7/6}+\ldots )^{2}   {\rm \ for\  some \ constant\ 
   } C\\
  &=& C(y^{7/2-2s/3 - 7/3}) +\ldots\\
  &=& Cy^{7/6 - 2s/3} + \ldots .
\end{eqnarray*}
So, we have shown that $f_x|_\phi = Cy^{7/6 - 2s/3} + \ldots \ \ {\rm when} \ \  s > 3/2$.  
 This gives rise to a second level to the tree representing $x=y^{3/2}+by^{s'} + h.o.t.$ by the line $3/2 <s'<\infty$.  On this line the order of $f_x$ is $\mathcal{O}(s')=(7-4s')/6$ for all $s' >3/2 $.  Note that this gives $\mathcal{O}(3/2) = 1/6$ so it is continuous with the order function on the first level. This second level line has no exceptional $s'$, it is actually an edge; so there are 3 edges on the first level and 1 more edge on the second level, which connects to the first level at the vertex $s=3/2$. This extends to the second quadrant by symmetry. Note $\mathcal{O}(s)>0$ on all of the first level, and $\mathcal{O}(s')\ge 0$ for $2/3 < s' \le 7/4$, so these combine to form four flat zones (two in each quadrant).  They are separated by the four open zones $x=\pm y^{3/2}+by^{s'} + h.o.t.$, $s'>7/4$, $b\ne 0$ (note these FI and FD zones surround the curve $x=\pm y^{3/2}$ which is $f=0$).
\smallskip

It is instructive to consider a much more general example of the type in Examples 1 and 2:

{\bf Example 3} 
Let 
\begin{eqnarray*}
z^n = \frac{p(x,y)}{q(x,y)}, &&{\rm with\ }n\ {\rm odd}\ ,\ p, q \in \R[x,y] {\rm \ polynomials},\\ 
&&{\rm and \ } q \ne 0\ {\rm except\ possibly\ at\ } \0.
\end{eqnarray*}

This includes both Examples 1 and 2.  

Take the partial derivative with respect to $x$ of both sides of our model to obtain
$$
nz^{n-1}z_x = \frac{p_xq - q_xp}{q^2}.
$$  
We conclude that 
$$
z_x =\left(\frac{f(x,y)}{g(x,y)}\right)^\frac{1}{n}\ {\rm for\ polynomials\ } f, g.
$$
We will substitute the arc 
$$
x = ay^s + h.o.t., \quad a \ne 0
$$
into the preceding expression for $z_x$ and compare the orders $\mathcal{O}(z_x)$ to  $\mathcal{O}(x) = s$.

Write
$$ f(x,y) = c_1x^{r_1}y^{s_1} + \ldots +c_mx^{r_m}y^{s_m}.$$
Then
\begin{eqnarray*} 
f(ay^s + h.o.t., y) &=& c_1a^{r_1}y^{sr_1+s_1} + h.o.t. +\ldots +c_ma^{r_m}y^{sr_m+s_m} + h.o.t. \\
\mathcal{O}(f(ay^s + h.o.t., y)) &=&\min(sr_1+s_1, \ldots, sr_m+s_m)
\end{eqnarray*}
\noindent  The subscript yielding the minimum may be different for different values of $s$, but for a fixed subscript the value is independent of $a$ and of the $h.o.t.$.  There may be two or more terms giving the same order for $s$, and the lowest order terms of the same order may cancel, but only for finitely many $s$.  The same is true for the denominator, hence for the quotient, hence for the $n^{th}$ root of the quotient.  

So, for all but finitely many $s$, the zones $Z^+_s = \{ ay^s + h.o.t., a> 0\}$ and $Z^-_s = \{ ay^s + h.o.t., a< 0\}$ will yield 
$$\mathcal{O}_s = \mathcal{O}(z_x(Z^{\pm}_s)).$$  
\noindent Call these the ``generic'' $s$.  The other $s$ are ``exceptional''.
Note that
\begin{eqnarray*}
FL_1 &=& \{ {\rm all\ arcs\ } \gamma\ {\rm with\ } \mathcal{O}(z_x(\gamma)) > 0\}\\
FL_2 &=& \{ {\rm all\ arcs\ } \gamma\ {\rm with\ } \mathcal{O}(z_x(\gamma)) = 0\}\\
FL &=& \{ {\rm all\ arcs\ } \gamma\ {\rm with\ } \mathcal{O}(z_x(\gamma)) \ge 0\}\\
FI\cup FD &=& \{ {\rm all\ arcs\ } \gamma\ {\rm with\ } \mathcal{O}(z_x(\gamma)) < 0\}.
\end{eqnarray*}
Between two exceptional $s$'s the interval of generic $s$ involve the order of
$$\left( \frac{c_ia^{r_i}y^{sr_i +s_i}}{c'_j a^{r_j'} y^{sr_j'+s_j'}} \right)^{\frac{1}{n}}$$
\noindent where the denominator comes from $g$ and some $i,j$. The displayed order 
is $$\frac{1}{n}((sr_i + s_i) - (sr_j' + s_j'))$$
\noindent and, in particular, affine in $s$.  So, there is at most one $s$ in this interval with $Z^{\pm}_s$ having order 0.  
This $Z^{\pm}_s$ is the only pair of $FL_2$ zones in this interval and they are closed.  

What happens at an exceptional $s_0$?

Then there are two or more terms in $f$ or in $g$ (or both) of the same minimal order, which may have cancellation thus increasing the order.  For example, suppose the numerator has minimal order terms
$$c_ia^{r_i}y^{s_0 r_i + s_i},\ \ c_ja^{r_j}y^{s_0 r_j+ s_j}\ \ ,c_ka^{r_k}y^{s_0 r_k + s_k}$$
of order
$$p=s_0r_i + s_i = s_0 r_j + s_j = s_0 r_k + s_k$$
and sum
$$(c_ia^{r_i} + c_ja^{r_j} + c_ka^{r_k})y^p.$$
Cancellation occurs if $a$ is a nonzero solution of the polynomial 
$$c_ia^{r_i} + c_ja^{r_j} + c_ka^{r_k} = 0.$$
Thus, there are only finitely many coefficients $a$ for  which the minimal terms cancel.  These are the ``exceptional" coefficients.  For all the ``generic" $a$, the set of all
$$\gamma_{s_o,a}: x = ay^{s_0} + h.o.t.$$
have a common order $\mathcal{O}(s_0)$ for $z_x(\gamma_{s_o,a})$ independent of $a$ (generic) and higher order terms.  Also
as $s \rightarrow s_0$ the order $\mathcal{O}(z_x(\gamma_s))\rightarrow \mathcal{O}(z_x(\gamma_{s_0},a)), a\ {\rm generic}.$

Now consider $a_0$ exceptional for $s_0$ exceptional.  The set of arcs $\{ a_0y^{s_0} + h.o.t.\}$ is an open zone, of order $s_0$, but the order of $z_x$ composed with these arcs will have different orders.  All these arcs are of the form
$$\gamma_{s_0,s',b}: x = a_0y^{s_0} + by^{s'} + h.o.t, \ s'> s_0.$$
The first term $x=a_0y^{s_0}$ substituted into $z_x$ will cause the lowest order terms in $z_x$ to cancel, but the next term $by^{s'}$ will cause the order of $z_x (\gamma_{s_0,s',b})$ to vary affinely toward $\mathcal{O}(z_x(\gamma_{s_0,a}))$ as $s\rightarrow s_0$ (see Example 2).  These orders will be independent of $b$ and the higher order terms.  As $s'$ increases, we may 
reach a new exceptional $s'$ at which the lowest terms cancel for exceptional choices of $b$, in which case we repeat the argument.   The next section will show what can happen in this case.

In general we will construct a tree of vertices and edges as follows.  For the first level of the tree, take the interval of values of $s$ (perhaps $0<s<\infty$ or $0<s\le 1$).  Each point $s$ on this interval represents the set of arcs $x=ay^s + h.o.t$, for all coefficients $a\ne 0$.  If $s$ is generic (not exceptional) then the order $O(f_x)$ will be independent of $a$ and $h.o.t$, and will be an affine function of $s$; it will also be independent of $a$ if $s$ is exceptional and $a$ varies over all generic (non-exceptional) coefficients for that $s$.  We let the  exceptional $s$ be the vertices, and the edges will be the connected components of the $s$ interval less the vertices. The ``order function" 
$\mathcal{O} (s)=\mathcal{O}(f_x(\gamma_{s,a}))$ will be affine on each edge and continuous at the vertices.  

All of the arcs are represented by the points on the tree constructed so far, except when we have an exceptional power $s_0$ with an exceptional coefficient $a_0$.  In that case the first level of the tree will be missing the arcs
$$\gamma_{s_0,s',b}: x = a_0y^{s_0} + by^{s'} + h.o.t, \ s'> s_0.$$
We connect to the point $s_0$ a line parametrized by $s'>s_0$; each point on this line represents the set of arcs $\gamma_{s_0,s',b}$ for all $b$.  Substituting into $f_x$ and computing the order we get that all but finitely many $s'$ are generic. The exceptional $s'$ represent vertices on this line with edges in between, forming the second level of the tree. 

For an exceptional $s'_0$ there can be finitely many exceptional coefficients; if $b_0$ is one such, then we need to add a new line connected to $s'_0$ with one point $s''$ representing the arcs $$\gamma_{s_0,s'_0,c}: x = a_0y^{s_0} + b_0y^{s'_0} +cy^{s''}+ h.o.t, \ s''> s'_0.$$  This line can then be broken down into part of the third level of the tree, and this process continues similarly.  Again see the next section for what can happen.

%

Each point on the graph represents a continuum of arcs such that the first varying term of the arcs have the same order, the order reciprocal is the width of the continuum of arcs; so on the first level of the graph this is the order of the arcs, on the second level this is the order of the second term of the arcs, etc.. 

All of the $FL$ zones occur when $\mathcal{O}(z_x\circ \gamma) \ge 0$, and this order function is continuous and piecewise affine on the graph, and so these zones  represent closed subsets of the tree.  The width of each $FL$ zone is realized by the width represented by one point in the closed subset realizing the width of the zone.  So the zone is a closed zone. 

In conclusion, all $FL$ zones for this example are closed.  

\section{Proof that $FL$ zones are closed}
So far we have been using elementary methods.  Employing some more advanced methods, we can prove a more general result about the FL zones.  In particular, consider the subanalytic preparation theorem due  to Parusinski (\cite{P}, 2001) and Van den Dries and Speissegger (\cite{VS}, 2002) as stated in \cite{BFGG}.

\smallskip

\noindent{\bf Theorem 3.2 of \cite{BFGG}}.  Let $f:( \R^2,\bf{0}) \rightarrow (\R, 0)$ be a definable and continuous function (we'll use subanalytic and continuous).  There exists a finite decomposition  $\mathcal{C}$ of $\R^2$ as a germ at $\bf{0}$ and for each $T \in \mathcal{C}$ there exists an exponent $\lambda \in \mathbb{F}$ (we'll use $\mathbb{Q} = \{ {\rm rationals} \}$) and definable functions $\theta, a: (\R, 0) \rightarrow \R$ and $u:(\R^2, \0) \rightarrow \R$ such that for $(x, y)\in T$ we have
$$
f(x,y) = (x - \theta(y))^\lambda a(y)u(x,y),\quad  | u(x, y) - 1 |  < \frac{1}{2}.
$$
By refining if necessary, we can further require that the set $\{ x = \theta(y) \}$ is either outside $T$ or on its boundary.

Consider $T\in \mathcal{C}$ as above for a continuous, subanalytic germ $f: \R^2,\0 \rightarrow \R, \bf{0}$.  For symplicity assume $T$ is in the first quadrant with left boundary arc $\alpha(y) = \Sigma a_i y^{r_i}$ and right
boundary arc $\beta(y) = \Sigma b_iy^{t_i}$. Let $\lambda$ and $a(y)$ be as in the theorem, and let $A=\mathcal{O}(a(y))$.

Suppose that $\theta(y) = \Sigma\theta_iy^{\sigma_1}$ is also in the first quadrant, either equal to $\alpha (y)$ or to the left of it.  Let $x(y) = \Sigma c_iy^{s_i}$ be an arc in  $T$. 
Thus $\sigma_1\ge r_1 \ge s_1\ge t_1$. By the theorem we have 
\begin{eqnarray*}
	\mathcal{O}(f(x(y),y)) &=& \lambda \mathcal{O}(\theta (y)-x(y)) + A\\
	&=&\lambda s_1 + A\ {\rm unless}\ \mathcal{O}(\theta (y)-x(y))> s_1 {\rm \ in\ which\ case}\  \\
	&&\quad \ \sigma_1 = r_1 = s_1 \ {\rm and}\  \theta_1 = a_1 = c_1.\\ 
\end{eqnarray*}
So, in the general case, we have a first order edge parameterized by all $s_1$, $r_1 \le s_1 \le t_1$ and
$	\mathcal{O}(f(x(y),y)) = \lambda s_1 + A$.
If $\sigma_1 = r_1 = s_1$ and $\theta_1 = a_1 = c_1$, then 
$$x(y) = a_1y^{r_1} + c_2y^{s_2} + \ldots$$
and $\mathcal{O}(\theta- x) = \min (\theta_2, s_2) = s_2,\ r_2\le s_2\  ({\rm and\ when}\ r_2 = s_2, \ a_2 \le c_2 ).$
This corresponds to a new edge (second order) representing these arcs with $$\mathcal{O}(f(x(y), y)) = \lambda s_2 + A.$$

If in addition $\sigma_2 = r_2 = s_2$ and $\theta_2 = a_2 = c_2$, then
\begin{eqnarray*}
	x(y) &=& a_1y^{r_1} + a_2y^{s_2} + c_2y^{s_3} + \ldots\\
	\mathcal{O}(\theta- x)&=& \min (\theta_3, s_3) = s_3, r_3\le s_3\ {\rm etc.}
\end{eqnarray*}
So, we have a third order edge with $\mathcal{O}(f(x(y),y)) = \lambda s_3 + A.$ The number of such edges depends on how close $\theta$ is to $\alpha$\ ! 

Suppose $\mathcal{O}(\theta-\alpha ) = \omega > 0$.  Then, there is an $I$ so that 
\begin{eqnarray*}
	&&\sigma_1= s_1, \quad \theta_1 = a_1\\
	&&\vdots\\
	&& \sigma_I =s_I, \quad \theta_I = a_I. \quad s_1 < \ldots < s_I <\omega\\
	&&\qquad {\rm and}\ \ \omega = \min (\sigma_{I+1}, s_{I+1})\\
\end{eqnarray*}
The last edge (closest to $\alpha$) represents the arcs
$$a_1y^{r_1} + \ldots +a_Iy^{r_I} + cy^s + \ldots,\, r_I < s.$$
For each $r_I < s \le \omega $, we have $\mathcal{O}(\theta- x) = s $, so $\mathcal{O}(f(x(y)),y) = \lambda s +a$.  

\noindent But for $s > \omega$, $\mathcal{O}(\theta- x) = \omega$ and $\mathcal{O}(f(\alpha (y), y)) = \lambda\omega + a$, so 
$$
\mathcal{O}(f(x(y),y)) \rightarrow \mathcal{O}(f(\alpha(y), y)) =\lambda \omega +a .
$$

\noindent Now, $\mathcal{O}(\theta-\alpha) = 0$ if and only if $\theta = \alpha$.  That is, if and only if $\theta$ is the left boundary of $T$.  In this case, there are an infinite number of edges and
$$
\mathcal{O}(f(x(y),y)) \rightarrow \mathcal{O}(f(\theta(y), y)) = +\infty
$$.

If $\theta$ is on or to the right of the right boundary $\beta$, we have a similar collection of order functions (except for the edge $[r_1, t_1]$ the order function is constant $\lambda t_1 + a$). However, to the right of this edge there are some number of edges as $x(y)$ gets closer to $\theta(y)$ just as in the previous case, and we have 
$$
\mathcal{O}(f(x(y),y)) \rightarrow \mathcal{O}(f(\beta(y), y))  .
$$

Since the order function on $T$ is continuous on $T$, including at the boundary arcs, it follows:

\begin{proposition} Let $f:( \R^2,\bf{0}) \rightarrow (\R, 0)$ be a subanalytic and continuous function. The order function of $f$ is continuous and piecewise affine on the entire order graph of $f$, taking values in $[0,+\infty]$.
\end{proposition}

If any boundary arc is a $\theta$ function, then $f=0$ on this edge and the  order function at the vertex representing that edge is $+\infty$ by definition, and the order function is considered continuous there because it approaches  $+\infty$ as we approach that vertex on the graph.

For every closed interval $I$ of the ``order graph'', there is a zone $Z(I)$ consisting of all arcs represented by $I$, the width of $I$ equals the largest width of a $Z(r)$ of arcs representing $r\in I$.  Each $Z(r)$ has multiple representatives having distance apart $w(Z(r))$, so $I$ by definition is closed.

Now assume $F(x,y,z)$ is an analytic function on a neighborhood of $0$.  Let $f(x,y)$ have graph contained in $F^{-1}(0)$ and $F^{-1}(0) - \Gamma f$ has dimension less than or equal to 1.  Then $f$ is semianalytic, so subanalytic.  We assume $f$ is continuous and $\pi |\Gamma f \rightarrow \R^2$ is a homeomorphism.

$F_x$ and $F_z$ are analytic, so $F_x^* = F_x (x, y, f)$ and $F_z^* = F_z(x, y, f)$ are continuous subanalytic maps $U\subset \R^2 \rightarrow \R$.  Restricting to any wedge $W$: $ax\le y\le bx$, $a<b, \, y\ge 0$, and applying the preparation theorem to each of $F_x^*$ and $F_z^*$ yields 
$\mathcal{C}(F_x^*)$ and $\mathcal{C}(F_z^*)$.  

Let $\mathcal{C}(f_x)$ be a common refinement of 
$\mathcal{C}(F_x^*)$ and $\mathcal{C}(F_z^*)$.
 
The set $V = \{ F^*_z =0\}$ consists of finitely many arcs such that $\Gamma f$ has a vertical tangent, and these arcs will be among the boundary  arcs of  $\mathcal{C}(f_x)$. Since $F^*_z$ and  $F^*_x$ have no common zeros except at $(0,0)$, $f_x=\frac{-F^*_x}{F_z^*} =\infty$, and we say $\mathcal{O}(f_x)=-\infty$ in this case. ($\mathcal{O}(f_x)=+\infty$ on arcs satisfying $f_x=0$).

The preparation theorem doesn't directly apply to $f_x$ because it is not continuous at $V$; nevertheles we get an order tree:
if $x(y)$ is an arc in $W$, then the order $x(y)$ for $F_x^*$ is some $\lambda_1r + A_1$ and for $F_z^*$ is some $\lambda_2r + A_2$ and so for $f_x $ $$\mathcal{O}(f_x(x(y)y)) =(\lambda_1r+A_1)-(\lambda_2r + A_2) = (\lambda_1 -\lambda_2)r + A_1 - A_2.$$ 

The FL zones are given by $\mathcal{O}(f_x) \ge 0$ and the FI and FD zones are where $\mathcal{O}(f_x) <0$.

Applying the Proposition and making some obvious adjustments, we get  

\begin{theorem} The order function $\mathcal{O}(f_x)$ is continuous and piecewise affine over any wedge  $ay\le x\le by$, $a<b,\, y\ge 0$; similarly over  $by\le x\le ay$, $a<b,\, y \le 0$.  Similarly $\mathcal{O}(f_y)$ is continuous and piecewise affine over any wedge  $ax\le y\le bx$, $a<b,\, x\ge 0$ or over  $bx\le y\le ax$, $a<b,\, x \le 0$.  All of the flat zones in these cases are closed.    
\end{theorem}

\section{Extending Theorem 6.6}

It is natural to ask whether Theorem 6.6 extends to the case when
$C_{\bf0}(X) = C \ne \R^2$ has an isolated singularity and is equidimensional of dimension 2 in 
$\R^3$.  In this case, let $N$ be a unit normal vector field on $C$. Let $f: C \rightarrow \R$ and $X = \ {\rm im}(f\cdot N),
\ \pi:X\rightarrow C$ be the inverse of $f\cdot N$.  It may happen that $\pi$ is singular in which case $T_xX$ 
is ``vertical'' containing the vector $N$.  The proof of the analogue of Theorem 6.6 in this setting should
be similar to that in this paper, but with the role of the variable $x$ being replaced by the angular variable $s$ in the
appropriate component of the tangent cone. In the neighborhood of any non-exceptional line in $C$, the geometry should be just like that in the case $C = \R^2$.   

It is also natural to ask whether our results extend to give an ambient bilipschitz classification of semialgebraic surfaces $X$ which are graphs $X = \Gamma f \subset \R^3$ of a semialgebraic function $f:U\rightarrow \R$ with isolated singularity at the origin and which have tangent cone $C = CX$ a plane. We believe that such an extension should be possible using the techniques employed here.

\end{document}